\documentclass{amsart}
\pdfoutput=1
\usepackage[top=3cm, bottom=3cm, left=3cm, right=3cm]{geometry} 
\usepackage{graphicx}
\usepackage{color}
\usepackage{amssymb}

\newcommand{\set}[1]{\{ #1 \} }
\newcommand{\wh}{\widehat}
\newcommand{\bd}{\partial}

\newtheorem{theorem}{Theorem}[section]

\theoremstyle{definition}
\newtheorem{definition}[theorem]{Definition}
\newtheorem{note}{Notation}

\theoremstyle{cor}

\newtheorem*{ack}{Acknowledgments}
\theoremstyle{remark}

\theoremstyle{example}

\numberwithin{equation}{section}

\begin{document}

\title{Tunnel One, Fibered Links}

\author{Matt Rathbun}
\address{Matt Rathbun, Department of Mathematics, Michigan State University, East Lansing, MI, 48824}
\email{mrathbun@math.msu.edu}


\begin{abstract}
Let $K$ be a tunnel number one, fibered link in $S^3$, with fiber $F$, and unknotting tunnel $\tau$. We show that $\tau$ can be isotoped to lie in $F$.

\end{abstract}

\maketitle

\section{Introduction and Motivation}
\label{sec:Introduction}

The study of fibered knots and links is as important today as ever. Giroux's correspondence between open book decompositions and contact structures \cite{GirGCD3DS} mingles classical fibered links with more modern contact geometry. Sutured manifold theory continues to reveal information about fibrations (see, for instance, \cite{SchaThoSKSI} and \cite{NiDSYF3M}). And fibered links are related to the newest advances in Floer homology, as Knot Floer homology detects fibered links (\cite{NiKFHDFK}) and Sutured Floer homology intersects both contact geometry and sutured manifold theory. 

Tunnel number one links are among the most-studied links. Much of the work on tunnel number one links revolves around trying to isotope the tunnel to sit nicely with respect to some additional structure in the $3$-manifold, including a hyperbolic metric (\cite{AdaUTH3M}, \cite{AdaReiUT2BKLC}, \cite{AkiNakSakSVGMOHDSWL}, \cite{CooLacPurLUT}), polyhedral decompositions (\cite{SakWeeECDHLC}, \cite{HeaSonUT-P237}), bridge decompositions (\cite{LacCAKTN1}, \cite{GodSchaThoLUT}), Seifert surfaces (\cite{SchaThoUTSS}), and fibrations (\cite{SakUTCDPTBOC}). These studies, and others, have led to the classification of tunnels for many classes of knots and links, including torus knots (\cite{BoiRosZieOHDTKERSFS}), satellite knots (\cite{MorSakOUTK}), non-simple links (\cite{EudUchNSLTN1}), $2$-bridge knots (\cite{MorSakOUTK}, \cite{KobCUT2BK}), and $2$-bridge links (\cite{JonC2GL}, \cite{MorOCTN1L}). 

Further, the Berge Conjecture states that if a knot admits a lens space Dehn surgery, then it is in one of the families of knots classified by John Berge. Many are working on this long-standing conjecture, with recent progress  contributed by Ozsv\'{a}th and Szab\'{o} \cite{OzsSzaKFHLSS}, Hedden \cite{HedFHBCKALSS}, Baker, Grigsby, and Hedden \cite{BakGriHedGDLSCKFH}, Saito \cite{SaiDS11KLS}, and Williams \cite{WilLSSPSTC}, among others. Yi Ni recently proved that if a knot admits a lens space surgery, then it is a fibered knot \cite{NiKFHDFK}. Additionally, all Berge knots are both fibered and tunnel number one, so further understanding tunnel one, fibered knots could have profound impacts on the Conjecture.

Jesse Johnson investigated genus $2$ Heegaard splittings of closed surface bundles over the circle \cite{JohSBG2HS}. This paper looks at the relationship between unknotting tunnels and fibrations for link complements.

\begin{theorem}
\label{theorem:TunnelNicelyInFiber}
Let $K$ be an oriented, fibered, tunnel number one link in $S^3$, with fiber $F$, and unknotting tunnel $\tau$. Then $\tau$ can be isotoped to lie in $F$.
\end{theorem}

\vspace{3 mm}

\begin{ack} 
The author would like to thank Abigail Thompson. The author was supported in part by NSF VIGRE Grants DMS-0135345 and DMS-0636297, and the RTG Grant DMS-0739208.
\end{ack}

\section{Background and Definitions}
\label{sec:Definitions}

\subsection{$3$-Manifolds}

\begin{note} Let $A$ be subset of a $3$-manifold $M$. We fix some notation. Let $n(A)$ denote a small open neighborhood of $A$ in $M$. If $F$ is a properly embedded surface in M, let $M | F = M \setminus n(F)$. If $S$ is the boundary of $M$, we will refer to $S| \partial F = S \setminus n(\bd F)$. For convenience, we will also sometimes refer to this as $S | F$.
\end{note}

\begin{definition} Let $F$ be a surface properly embedded in a $3$-manifold $M$. $F$ is said to be \emph{compressible} if there exists a disk $D$ embedded in $M$ with $\bd D = D \cap F$ an essential curve in $F$. Then $D$ is called a \emph{compressing disk} for $F$. If $F$ is not compressible, and is not a $2$-sphere, then it is called \emph{incompressible}. $F$ is said to be \emph{boundary compressible} if there exists a disk $D$ embedded in $M$ with $D \cap F = \alpha \subset \bd D$, $D \cap \bd M = \beta \subset \bd D$, $\alpha$ is an essential arc in $F$, and $\alpha \cap \beta = \bd \alpha = \bd \beta$, $\alpha \cup \beta = \bd D$. In this case, $D$ is called a \emph{boundary compression disk}. If $F$ is not boundary compressible, it is called \emph{boundary incompressible}.
\end{definition}

\begin{definition} A \emph{compression body} $V$ is the result of taking the product of a surface with $[0, 1]$, and attaching $2$-handles along $S \times \{0\}$, and then $3$-handles along any resulting $2$-sphere components. $S \times \{1\}$ is called $\partial_+ V$, and $\partial V \setminus \partial_+ V$ is called $\partial_- V$. A \emph{handlebody} is a compression body where $\partial_- V = \emptyset$. A \emph{Heegaard splitting} is a triple $(S, V, W)$, where $S$ is a surface, $V$ and $W$ are compression bodies, $\partial_+ V = \partial_+ W = S$, and $M = V \cup_{S} W$.
\end{definition}

\begin{definition} Let $K$ be a knot in a $3$-manifold $M$, and let $\lambda$ be an essential closed curve in $\bd n(K)$. Let $M'$ be the manifold obtained from $M$ by removing $n(K)$, and attaching a solid torus $S^1 \times D^2$ to $M \setminus n(K)$ via a homeomorphism of the boundaries such that $\set{pt.} \times \bd D^2$ is identified with the curve $\lambda$. Then $M'$ is said to be the result of \emph{$\lambda$-sloped Dehn surgery} on $M$.  
\end{definition}

\subsection{Tunnels}
 
 \begin{definition}
A link $L$ in a $S^3$ is called a \emph{tunnel number one} link if there exists an arc $\tau$ properly embedded in $S^3 \setminus n(L)$ such that $S^3 \setminus n(L \cup \tau)$ is a handlebody. Then $\tau$ is called a \emph{tunnel} for $L$.
 \end{definition}
 
 Observe that the complement of a tunnel number one link has a genus $2$ Heegaard splitting. Also, note that a tunnel one link has at most $2$ components, and if it has two components, then any tunnel must have one endpoint on each component.
 
 More generally, a knot is \emph{tunnel number $n$} if $n$ is the smallest number such that there exists a collection of arcs $\set{\tau_1, \dots, \tau_n}$ such that $S^3 \setminus n(L \cup \tau_1 \cup \cdots \cup \tau_n)$ is a handlebody. 
 
\subsection{Fibered Links}

\begin{definition} Let $L \subset S^3$ be a link. A \emph{Seifert surface} for $L$ is a compact, connected, orientable surface $F$ embedded in $S^3$ such that $\bd F = L$. 
\end{definition}

\begin{definition} A map $f : E \to B$ is a \emph{fibration} with \emph{fiber} $F$ if for every point $p \in B$, there is a neighborhood $U$ of $p$ and a homeomorphism $h : f^{-1}(U) \to U \times F$ such that $f = \pi_1 \circ h$, where $\pi_1 : U \times F \to U$ is projection to the first factor. $E$ is called the \emph{total} space, and $B$ is called the \emph{base} space. Each set $f^{-1}(b)$ is called a \emph{fiber}, and is homeomorphic to $F$. 
\end{definition}
 
\begin{definition} A link $L \subset S^3$ is said to be \emph{fibered} if there is a fibration of $S^3 \setminus n(L)$ over $S^1$, and the fibration is well-behaved near $L$. That is, each component $L_i$ of $L$ has a neighborhood $S^1 \times D^2$, with $L_i \cong S^1 \times \set{0}$ such that $f|_{S^1 \times (D^2 \setminus \set{0})}$ is given by $(x, y) \to \frac{y}{|y|}$.
\end{definition}

Each fiber of a fibered link is a Seifert surface for the link. The complement of a fibered link is foliated by copies of this Seifert surface. Cutting along one of these Seifert surface produces a surface cross the interval.

\begin{definition} Let $K$ be a fibered link in $S^3$. Then $S^3 \setminus n(K)$ can be obtained from $F \times I$, with $F$ a fiber, by identification $(x, 0) \sim (h(x), 1)$, for $x \in F$, where $h: F \to F$ is an orientation-preserving homeomorphism which is the identity on $\bd F$. We call $h$ a \emph{monodromy} map.
\end{definition}

\subsection{Theorems}
 
Our starting point is a theorem proven by Scharlemann and Thompson.

\begin{theorem}[\cite{SchaThoUTSS}] \label{theorem:KnotDisjoint} Suppose $K$ is a knot in $S^3$, and $\tau$ an unknotting tunnel for $K$. Then $\tau$ may be slid and isotoped until it is disjoint from some minimal genus Seifert surface for $K$.
\end{theorem}

The proof consists of arranging $K$, $\tau$, and a compressing disk for $S^3 \setminus n(K \cup \tau)$ in some minimal fashion, and showing that if $K \cap \tau \neq \emptyset$, then this would be a contradiction to those minimality assumptions. The result still holds for fibered links. 
 
\begin{theorem} \label{theorem:Disjoint} Suppose $K$ is an oriented, fibered link, and $\tau$ is an unknotting tunnel for $K$. Then $\tau$ may be slid and isotoped until it is disjoint from a fiber of $K$. 
\end{theorem}

Our proof will largely mimic \cite{SchaThoUTSS}.

\begin{proof}
By \ref{theorem:KnotDisjoint}, an unknotting tunnel can be isotoped and slid to be disjoint from a minimal genus Seifert surface for a knot. But in a fibered knot complement, a fiber is the unique minimal genus Seifert surface, so the result follows. Henceforth, let us assume that $K$ is a two-component link, and let the two components of $K$ be $K_1$ and $K_2$. Observe that $\tau$ has one endpoint on each of the components of $K$. Choose a fiber $F$, and slide and isotope $\tau$, so as to minimize the number of intersections between $\tau$ and $F$. Our goal will be to prove that $\tau \cap F = \emptyset$.

Suppose, to the contrary, that after the slides and isotopies above, $\tau \cap F$ is non-empty. Let $E$ be an essential disk in the handlebody $S^3 \setminus n(K \cup \tau)$, chosen to minimize the number $|E \cap F|$ of components in $E \cap F$. If $|E \cap F| = 0$, then the incompressible $F$ would lie in a solid torus, namely (a component of) $S^3 \setminus n(K \cup \tau \cup E)$, and so be an annulus. The only fibered link with fiber an annulus is the Hopf link, in which case the result holds. So we may assume that $|E \cap F| > 0$. Furthermore, since $F$ is incompressible, we may assume that $E \cap F$ consists entirely of arcs.

Let $e$ be an outermost arc of $E \cap F$ in $E$, cutting off a subdisk $E_0$ from $E$. The arc $e$ is essential in $F \setminus \tau$, for otherwise we could find a different essential disk intersecting $F$ in fewer components. Let $f = \bd(E_0) \setminus e$, an arc in $\bd n(K \cup \tau)$ with each end either on a longitude $\bd F \subset \bd n(K)$ or a meridian disk of $\tau$ corresponding to a point of $\tau \cap F$.

If a meridian of $\tau$ is incident to exactly one end of $f$, then we can use $E_0$ to describe a simple isotopy of $\tau$ which reduces the number of intersections between $\tau$ and $F$. 

If no meridian of $\tau$ is incident to an end of $f$, then both ends of $f$ lie on $\bd F \subset \bd n(K)$. If the interior of $f$ runs over $\tau$, we have finished, for $f$ is disjoint from $F$. If the interior of $f$ lies entirely in $\bd n(K)$, and $e$ is essential in $F$, then $E_0$ would be a boundary compression disk for $F$, contradicting the minimality of the genus of $F$. If the interior of $f$ lies entirely in $\bd n(K)$ and $e$ is inessential in $F$, then the disk $D_0$ that it cuts off from $F$ necessarily contains points of $\tau$ (since $e$ is essential in $F \setminus \tau$). But then we could replace $D_0$ by $E_0$. If the loop formed by $f$ and $\bd D_0 \setminus e$ form a trivial loop on the torus of, say, $\bd n(K_1)$, then the new surface would, again, be a Seifert surface for $K$, consistent with the orientation of $K$ (and so be a fiber), but with fewer points of intersection $F \cap \tau$. If the loop formed by $f$ and $\bd D_0 \setminus e$ is essential in $\bd n(K_1)$, then the original disk $E_0$ could be slid across $D_0$ to show that $K_1$ is unknotted. But the interior of the disk $D_0$ is disjoint from $K_2$, so $K$ must be a split link, but split links do not fiber.

The only remaining possibility is that both ends of $f$ lie on the same meridian of $\tau$. In this case, $e$ forms a loop in $F$, and the ends of $f$ adjacent to $e$ both run along the same sub-arc $\tau_0$ of $\tau$. Since $f$ is disjoint from $F$, $\tau_0$ terminates on, say, $\bd n(K_1)$.

Then since the interior of $f$ is disjoint from $F$, $f$ must intersect $\bd n(K_1)$ either in an inessential arc in the torus or in a longitudinal arc. The former case is impossible, because the inessential disk cannot contain the other end of $\tau$ (since the other end of $\tau$ is on $\bd n(K_2)$), and so the disk can be isotoped away, reducing $|E \cap F|$. It follows that $f$ intersects the torus $\bd n(K_1)$ in a longitudinal arc. Then, $n(\tau_0 \cup E_0)$ is a thickened annulus $A$, defining a parallelism in $S^3$ between $K_1$ and the loop $e$ on $F$. Now, the boundary component of $A$ on $\bd n(K_1)$ can be slid across $\bd n(K_1)$, away from $e$, onto $F$, parallel to $\bd F$ in $F$. Since $K$ is a fibered link, the image of $A$, call it $A'$, is a product annulus in $S^3 \setminus n(K \cup F) \cong F \times I$. But then this demonstrates that $e$ itself is, in fact, parallel to $\bd F$ in $F$. Then, substituting $A$ for the annulus between $e$ and $\bd F$ in $F$ would create a Seifert surface of the same genus, still consistent with the orientation of $K$, thus a fiber, but with fewer intersections with $\tau$, a contradiction.

\end{proof}

Another theorem that we will find useful is also given by Scharlemann and Thompson. Yi Ni proves a more general result in \cite{NiDSYF3M}, though we will not need it here.

\begin{theorem}[\cite{SchaThoSKSI}] \label{theorem:Surgery} Suppose $F$ is a compact orientable surface, $L$ is a knot in $F \times I$, and $(F \times I)_{surg}$ is the $3$-manifold obtained by some non-trivial surgery on $L$. If $F \times \set{0}$ compresses in $(F \times I)_{surg}$, then $L$ is parallel to an essential simple closed curve in $F \times \set{0}$. Moreover, the annulus that describes the parallelism determines the slope of the surgery.
\end{theorem}

The proof relies on sutured manifold theory, and a theorem of Gabai \cite{GabSKST}. Gabai proves the result for an annulus cross the interval. The idea of Scharlemann and Thompson's proof is to find product disks or annuli in $(F \times I)$ disjoint from the knot, and cut along these product pieces to reduce the complexity of the surface in question. This, with some additional work, allows them to apply the results of Gabai.

\section{Pushing a Tunnel into a Fiber}
\label{sec:PushingTunnelIntoFiber}

\begin{proof}[Proof of Theorem \ref{theorem:TunnelNicelyInFiber}]

By Proposition \ref{theorem:Disjoint}, $\tau$ can be isotoped and slid to be disjoint from a fiber. Let $F = F' \setminus n(K)$. Cut $S^3 \setminus n(K)$ along $F$, to obtain $N \cong F \times I$, a handlebody. Then $\tau \subset N$.

Now, as $\tau$ is an unknotting tunnel, there exists a compressing disk for $\partial n(K \cup \tau)$ in $S^3 \setminus n(K \cup \tau)$, say $D'$. Note that $D' \cap F \neq \emptyset$, for otherwise $F$ would be an essential surface in the solid torus $(S^3 \setminus n(K \cup \tau))|D'$, and thus a disk.

Consider $D' \cap F$. Since $F$ is incompressible and $N$ is irreducible, by standard innermost disk arguments we may assume there are no simple closed curves of intersection. Let $\alpha$ be an arc of intersection which is outermost in $D'$, cutting off a subdisk $D$. Then, $D$ is a disk in $N$ with boundary consisting of three types of arcs: a single essential arc in $F = F \times \set{0}$, $\alpha$; (several) arcs in $\partial n(K)$, call them $\nu_i$; and (several) arcs in $\partial n(\tau)$, $\lambda_j$. We may assume that every arc of $D \cap \partial n(\tau)$ is an essential spanning arc of the annulus $\partial n(\tau)$, for trivial arcs can be removed by isotopy.

Now, consider the double of $N$, along $\bd n(K) | F$. In other words, let $\wh{N}$ be the result of gluing two copies of $N$ together by the identity along $\bd F \times I$. Similarly, let $\wh{\tau}$ be the result of gluing two copies of $\tau$, one in each copy of $N$, along the boundary points; let $\wh{D}$ come from two copies of $D$, one in each copy of $N$, glued along the $\nu_i$; and let $\wh{\alpha}$ come from two copies of $\alpha$ in the same way. 

Then $\wh{D}$ is a planar surface with one boundary component corresponding to $\wh{\alpha}$, and several components coming from $\wh{\lambda_j}$, the doubles of $\lambda_j$ (see Figure \ref{Figure:Dhat}).

\begin{figure}[tb]
\begin{center}
\includegraphics[width=2in]{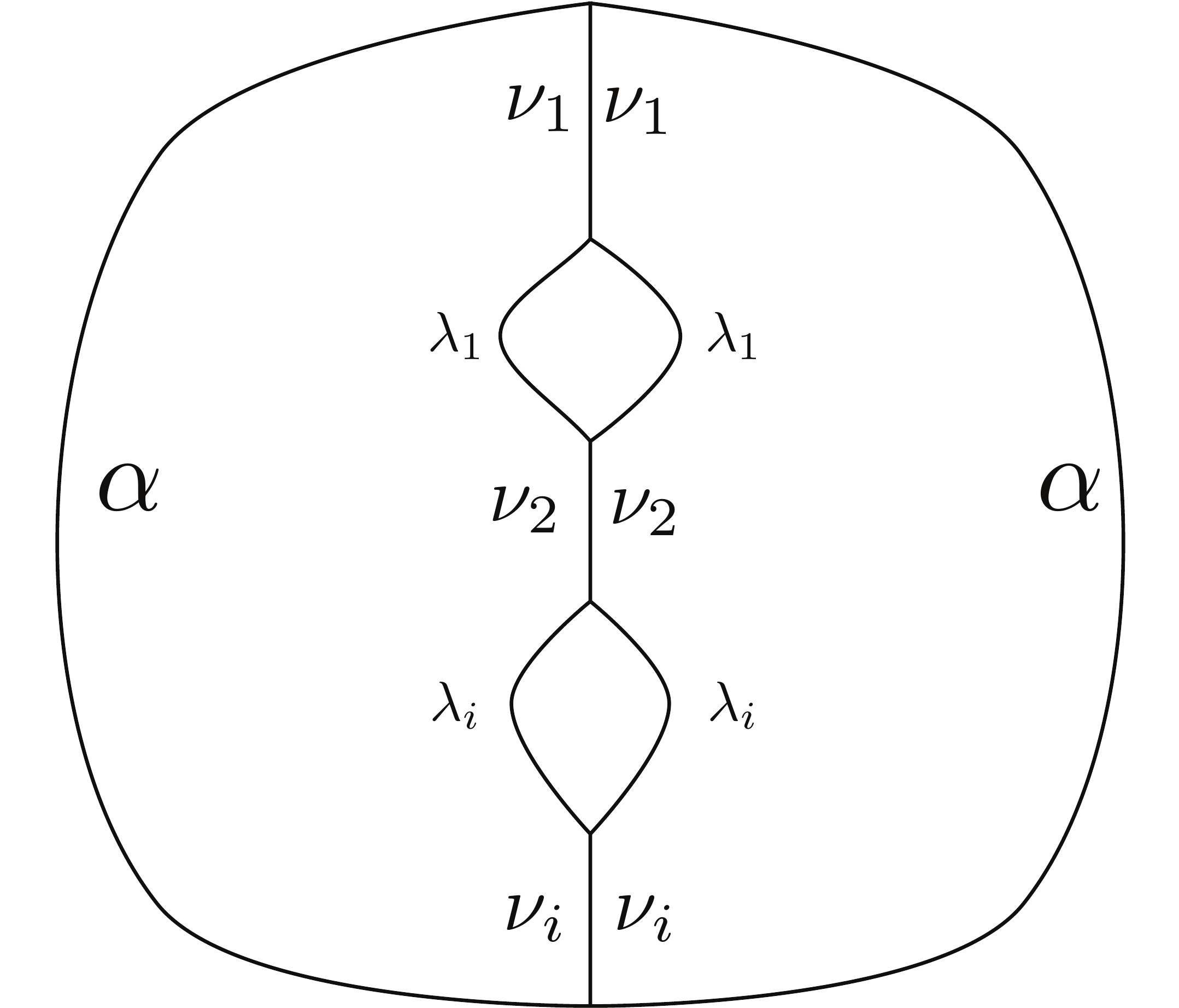}
\caption{$\wh{D}$}
\label{Figure:Dhat}
\end{center}
\end{figure}

Then, $\cup_j \wh{\lambda_j}$ is a collection of (parallel) simple closed curves on the torus $\bd n(\wh{\tau})$. Call the slope determined by these curves $\lambda$. If we perform $\lambda$-surgery on $\wh{\tau}$, the result is to cap off $\wh{D}$ with disks. Since $\alpha$ was essential in $F$, $\wh{\alpha}$ is essential in $\wh{F}$, so our capped off surface is a compression disk for $\wh{F}$ in $\wh{N} \cong \wh{F} \times I$. 

By Theorem \ref{theorem:Surgery}, $\wh{\tau}$ is parallel to an essential closed curve in $\wh{F} \times \set{0}$. That is, there exists an annulus $A$ properly embedded in $\wh{F} \times I$ with one boundary component on $\wh{F} \times \set{0}$, say $\psi$, and the other boundary component on $\bd n(\wh{\tau})$, parallel to $\wh{\tau}$, say $\phi$.

Since $\phi$ is parallel to $\wh{\tau}$, it must be a longitude of $\bd n(\wh{\tau})$, and in particular, $\nolinebreak{|\phi \cap (\bd F \times I) | = 2}$. So there are only two possibilities for arcs of intersection between $A$ and $\bd F \times I$ incident to $\phi$. Either, there is one arc of intersection which is trivial in $A$, or there are two arcs of intersection, both of which are essential in $A$ (see Figure \ref{Figure:Arcs}). The former case is impossible, because then the subdisk of $A$ cut off by the arc would show that $\tau$ was parallel into $\bd n(K)$, which would imply that $K$ was trivial. Therefore, there are exactly two arcs of $A \cap (\bd F \times I)$, both of which are essential in $A$.

\begin{figure}[tb]
\begin{center}
\includegraphics[width=4in]{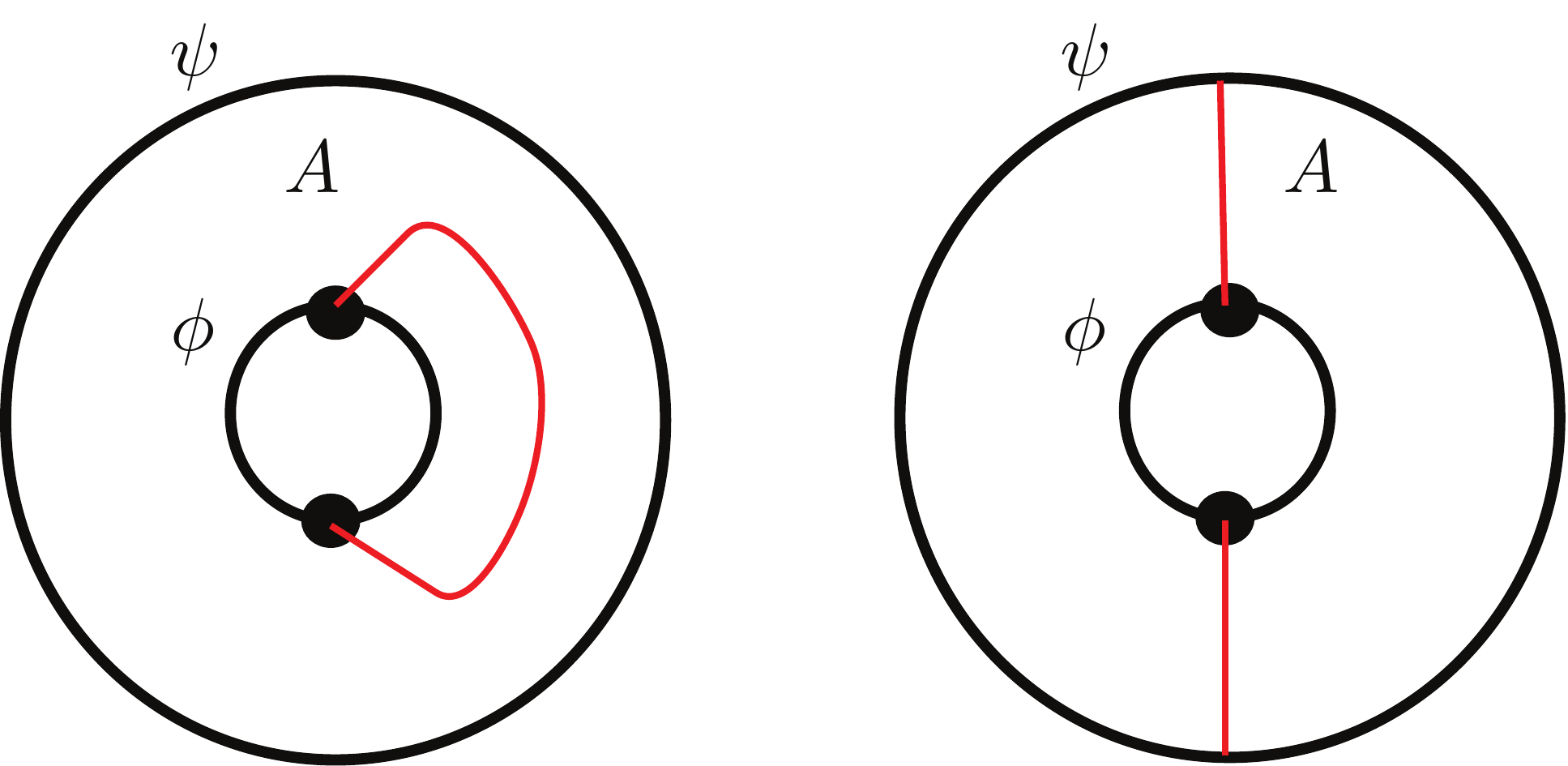}
\caption{Arcs of $A \cap \bd F \times I$ incident to $\phi$.}
\label{Figure:Arcs}
\end{center}
\end{figure}

If there were trivial arcs incident to $\psi$, then an outermost such arc in $A$ would give rise to a boundary compression for $F \times \set{0}$ in $S^3 \setminus n(K)$. This is impossible as well, so $\bd F \times I$ intersects $A$ in precisely two essential arcs, with no trivial arcs. Cutting $A$ along these arcs provides a parallelism between $\tau$ and the arc $\psi \cap (F \times \set{0}) \subset \wh{F} \times \set{0}$. Thus, $\tau$ can be isotoped to lie in the fiber.

\end{proof}


\end{document}